\documentclass[review]{elsarticle}

\usepackage{lineno,hyperref}
\modulolinenumbers[5]
\usepackage{amssymb}
\usepackage{amsmath}
\usepackage{hyperref}
\usepackage{setspace}
\usepackage{algorithm}
\usepackage{algorithmic}
\usepackage{booktabs}
\usepackage{graphicx,fancyhdr}
\usepackage{graphics,color}
\usepackage{subfigure}
\usepackage{booktabs}
\usepackage{graphicx}
\usepackage{amsthm}
\newtheorem{theorem}{Theorem}[section]
\newtheorem{lemma}{Lemma}[section]

\newtheorem{remark}{Remark}[section]
\newtheorem{definition}{Definition}[section]

\newtheorem{example}{Example}[section]

\journal{Journal of \LaTeX\ Templates}

\begin{document}

\begin{frontmatter}

\title{Central local discontinuous Galerkin method for the space fractional diffusion equation}

\author[mymainaddress]{Jing Sun}
\ead{sunj2015@lzu.edu.cn}

\author[mymainaddress]{Daxin Nie}
\ead{ndx1993@163.com}

\author[mymainaddress]{Weihua Deng\corref{mycorrespondingauthor}}
\cortext[mycorrespondingauthor]{Corresponding author}
\ead{dengwh@lzu.edu.cn}

\address[mymainaddress]{School of Mathematics and Statistics, Gansu Key Laboratory of Applied Mathematics and Complex Systems, Lanzhou University, Lanzhou 730000, P.R. China}

\begin{abstract}
	
	This paper provides the semi-discrete scheme by the central local discontinuous Galerkin method for space fractional diffusion equation on two sets of overlapping cells, and then we give the stability analysis and error estimates for the scheme. Lastly, we verify the effectiveness of the proposed scheme by the one- and two-dimensional numerical experiments.
	
\end{abstract}


\end{frontmatter}


\section{Introduction}
\label{}
 L\'evy processes (and the related ones) are effective microscopic models for anomalous diffusions, and their corresponding macroscopic models governing some interesting/useful statistical observables are generally the integral-differential equations \cite{Deng2018}, some special cases of which are so-called fractional partial differential equations (FPDEs). Anomalous diffusions are ubiquitous in the natural world, such as, physics \cite{Metzler2000,Metzler20002}, biochemistry \cite{Yuste2004}, hydrology \cite{Liu2004}, etc.
The nonlocality of the fractional operators makes it a challenging work to find the exact solution of FPDEs, hence how to efficiently get the numerical solutions of FPDEs becomes a big deal. So
far, some progresses have been made for numerically solving FPDEs, such as, finite difference, finite element, spectral method, discontinuous Galerkin (DG), local  discontinuous Galerkin (LDG), etc (see \cite{Basu2012,Bu2014,Chen2014,Cui2009,Deng2013,Liu2004,Mao2016,Xu2014,Zhang2010}).

Recently, Ref. \cite{Liu2007} introduces a central DG method to solve the system of conservation laws based on the central scheme \cite{Nessyahu1990}, which needs to use the redundant representation of the solution on overlapping cells depicted in \cite{Cockburn1998}. Compared with traditional DG, the central DG uses  duplicative information on overlapping cells, which  doesn't need numerical fluxes any more to provide the information at the interface of cells, since the evaluation of the solution at the interface is in the middle of the staggered mesh; moreover, this scheme covers the two main advantages of DG method, i.e., the  flexibility of grid subdivision and excellent parallel efficiency. Afterwards, Ref. \cite{Liu2008} gives the stability and error analysis for the central DG scheme. Reference \cite{Cockburn1998} provides the LDG finite element method  for the time dependent convection-diffusion systems. Reference \cite{Deng2013} introduces LDG method with suitable fluxes to solve the spatial fractional diffusion equation and then \cite{Qiu2015} generalizes it to the two-dimensional cases. Later, Ref. \cite{Liu2011} introduces a central LDG based on central DG and LDG  for traditional diffusion equations, which gives two versions of the central LDG schemes and provides the stability and error analysis.  To the best of our knowledge, it seems that there is no work solving the space fractional diffusion equation by the central LDG, especially for higher dimensional cases.

In this paper, we use the central  LDG method  to solve the two-dimensional space fractional diffusion equation with the homogeneous Dirichlet boundary condition, i.e., find $u=u(x,y,t)$ satisfying
\begin{equation}\label{equation2D}
\left\{
\begin{aligned}
&\frac{\partial u(x,y,t)}{\partial t}-d_1(~_{-\infty}D_{x}^{\alpha}u+~_{x}D_{\infty}^{\alpha}u)
\\&~~~~~~~~~~~~~-d_2(~_{-\infty}D_{y}^{\beta}u+~_{y}D_{\infty}^{\beta}u)=f(x,y,t)~~~(x,y,t)\in \Omega\times[0,T],\\
&u(x,y,0)=g(x,y)~~~~~~~~~~~~~~~~~~~~~~~~~~~~~~~~~~~~~~~(x, y)\in \Omega,\\
&u(x,y,t)=0~~~~~~~~~~~~~~~~~~~~~~~~~~~~~~~~~~~~~~~~~~(x, y)\in \mathbb{R}^2\backslash \Omega, ~~t\in [0,T],
\end{aligned}
\right.
\end{equation}
where $\Omega=(0,1)\times(0,1)$, $1<\alpha, \beta<2$, and $d_1, d_2>0$ are two constants;  $~_{-\infty}D_{x}^{\alpha}u$ and $_{x}D_{\infty}^{\alpha}u$ are the left- and right-sided Riemann-Liouville derivatives, respectively. Since ${\rm supp} (u)\subset\Omega$, we have $_{-\infty}D_x^{\alpha}u(x,y,t)$ =$~_0D_x^{\alpha}u(x,y,t)$, $_{x}D_\infty^{\alpha}u(x,y,t)$ =$~_xD_1^{\alpha}u(x,y,t)$, $(x,y)\in \Omega$. For convenience, we set $d_1=d_2=1$ in this paper.
Here, we solve Eq. (\ref{equation2D}) numerically by using the duplicative information on overlapping cells and avoid using the numerical flux to define the interface values of the solution. Meanwhile, in order to guarantee the stability and convergence of our scheme, we modify the LDG scheme in \cite{Qiu2015} according to the theoretical prediction and the feature of the central LDG scheme in \cite{Liu2011}.

The remainder of this paper is arranged as follows. Sec. 2 provides some necessary definitions and properties of the fractional derivatives and integrals. In Sec. 3, we give the spatial semi-discrete scheme by the central LDG method. Sec. 4 presents the stability analysis and error estimates of the designed scheme. Sec. 5 discusses the implementation details of numerical simulations. We conclude the paper with discussions in the last section.


\section{Preliminaries}
\label{}
In this section, we recall some definitions and properties of fractional derivatives, fractional integrals, and fractional Sobolev spaces.
\begin{definition}[\cite{Podlubny1999}]\label{def1}
	The left- and right-sided Riemann-Liouville fractional integrals of order $\alpha~(\alpha>0)$ are defined by
	\begin{equation*}
	_{-\infty}I^{\alpha}_xu(x)=\frac{1}{\Gamma(\alpha)}\int^x_{-\infty}(x-\xi)^{\alpha-1}u(\xi)d\xi,
	\end{equation*}
	\begin{equation*}
	_xI^{\alpha}_{\infty}u(x)=\frac{1}{\Gamma(\alpha)}\int^{\infty}_x(\xi-x)^{\alpha-1}u(\xi)d\xi.
	\end{equation*}
\end{definition}
\begin{definition}[\cite{Podlubny1999}]\label{def2}
	The left- and right-sided Riemann-Liouville fractional derivatives of order $\alpha~(\alpha>0)$ are defined by
	\begin{equation*}
	_{-\infty}D^{\alpha}_xu(x)=\frac{1}{\Gamma(n-\alpha)}\frac{d^n}{dx^n}\int^x_{-\infty}(x-\xi)^{n-\alpha-1}u(\xi)d\xi,
	\end{equation*}
	\begin{equation*}
	_xD^{\alpha}_{\infty}u(x)=\frac{(-1)^n}{\Gamma(n-\alpha)}\frac{d^n}{dx^n}\int^{\infty}_x(\xi-x)^{n-\alpha-1}u(\xi)d\xi,
	\end{equation*}
	where $n-1\leq\alpha<n$, $n\in \mathbb{N}$. 
\end{definition}

\begin{definition}[\cite{Ervin2006}]
	Let $\alpha>0$. Define the semi-norm
	\begin{equation*}
	|u|_{J^\alpha_L(\mathbb{R})}:=\|~_{-\infty}D^\alpha_xu\|_{L^2(\mathbb{R})}
	\end{equation*}
	and norm
	\begin{equation*}
	\|u\|_{J^\alpha_L(\mathbb{R})}:=\left (\|u\|^2_{L^2{(\mathbb{R})}}+|u|^2_{_{J^\alpha_L(\mathbb{R})}}\right )^{1/2},
	\end{equation*}
	and $ J^\alpha_{L}(\mathbb{R}) $ denotes the closure of $C_0^\infty(\mathbb{R})$ with respect to $ \|\cdot\|_{J^\alpha_L(\mathbb{R})} $.
\end{definition}

\begin{definition}[\cite{Ervin2006}]
	Let $\alpha>0$. Define the semi-norm
	\begin{equation*}
	|u|_{J^\alpha_R(\mathbb{R})}:=\|~_{x}D^\alpha_{\infty}u\|_{L^2(\mathbb{R})}
	\end{equation*}
	and norm
	\begin{equation*}
	\|u\|_{J^\alpha_R(\mathbb{R})}:=\left (\|u\|^2_{L^2{(\mathbb{R})}}+|u|^2_{_{J^\alpha_R(\mathbb{R})}}\right )^{1/2},
	\end{equation*}
	and let $ J^\alpha_{R}(\mathbb{R}) $ denote the closure of $C_0^\infty(\mathbb{R})$ with respect to $ \|\cdot\|_{J^\alpha_R(\mathbb{R})} $.
\end{definition}

\begin{definition}[\cite{Ervin2006}]
	Let $\Omega=(a,b)$ be a bounded open subinterval of $\mathbb{R}$. Then we define the spaces $J^\alpha_{L,0}(\Omega),J^\alpha_{R,0}(\Omega)$ as the closures of $C^\infty_0(\Omega)$ under their respective norms.
\end{definition}

\begin{definition}[\cite{Ervin2006}]
	For $0\leq\alpha<\infty$, we define the space
	\begin{equation*}
	H^{\alpha}(\mathbb{R}^{n}):=\{u\mid u \in L^{2}(\mathbb{R}^{n}),~(1+|\omega|^{2})^{\frac{\alpha}{2}}\hat{u}(\omega)\in L^{2}(\mathbb{R}^{n})\}
	\end{equation*}
	with the norm
	\begin{equation*}
	\|u\|_{H^{\alpha}(\mathbb{R}^n)}:=(\|u\|^{2}_{L^2{(\mathbb{R}^n)}}+|u|^2_{H^\alpha(\mathbb{R}^{n})})^{1/2} ~~~\forall~ u\in H^{\alpha}(\mathbb{R}^{n}),
	\end{equation*}
	where $|u|_{H^\alpha(\mathbb{R}^{n})}:=\||\omega|^{\alpha}\hat u\|_{L^2(\mathbb{R}^n)}$, and $\hat{u}$ denotes the Fourier transform of $u$.
\end{definition}
\begin{definition}[\cite{Ervin2007,Li2009}]
	For $\Omega\subset \mathbb{R}^{n}$, we define
	\begin{equation*}
	H^{\alpha}(\Omega):=\{v|_{\Omega}\ |v\in H^{\alpha}(\mathbb{R}^{n})\}
	\end{equation*}
	with the norm
	\begin{equation*}
	\|v\|_{H^{\alpha}(\Omega)}:=\inf_{\substack{
			\tilde{v}\in H^{\alpha}(\mathbb{R}^{n})\\
			\tilde{v}\mid_{\Omega}=v}} \|\tilde{v} \|_{H^\alpha(\mathbb{R}^{n})}~~~~\forall~ v\in H^\alpha(\Omega).
	\end{equation*}
	Furthermore, define $H_{0}^{\alpha}(\Omega)$ as the closure of $C^\infty_0(\Omega)$ with respect to $\|\cdot\|_{H^{\alpha}(\Omega)}$; for $\Omega=(a,b)\subset \mathbb{R}$, define $~_0H^{\alpha}(\Omega)$ as the closure of $~_0C^{\infty}(\Omega)$ with respect to $\|\cdot\|_{H^{\alpha}(\Omega)}$, where
	\begin{equation*}
	~_0C^\infty(\Omega)=\{v|v\in C^{\infty}(\Omega)\  with\  compact\  support \ in\  (a,b]\}.	\end{equation*}
\end{definition}

\begin{remark}
	It is known (see \cite{Li2009,Lions1972}) that
	\begin{equation*}
	~_0H^{\alpha}(\Omega)=H^{\alpha}(\Omega)=H_0^{\alpha}(\Omega),\quad if\ 0<\alpha<\frac{1}{2}.
	\end{equation*}
\end{remark}

\begin{definition}[\cite{Ervin2007}]
	Let $\alpha>0$. Define the semi-norms
	\begin{equation*}
	\begin{aligned}
	&|u|_{J^\alpha_{L,0\pi}(\mathbb{R}^2)}:=\|~_{-\infty}D^\alpha_xu\|_{L^2{(\mathbb{R}^2)}},\\
	&|u|_{J^\alpha_{L,\pi/2}(\mathbb{R}^2)}:=\|~_{-\infty}D^\alpha_yu\|_{L^2{(\mathbb{R}^2)}},\\
	&|u|_{J^\alpha_{L,\pi}(\mathbb{R}^2)}:=\|~_xD^\alpha_\infty u\|_{L^2{(\mathbb{R}^2)}},\\
	&|u|_{J^\alpha_{L,3\pi/2}(\mathbb{R}^2)}:=\|~_yD^\alpha_\infty
	u\|_{L^2{(\mathbb{R}^2)}},
	\end{aligned}
	\end{equation*}
	and norms
	\begin{equation*}
	\begin{aligned}
	&\|u\|_{J^\alpha_{L,0\pi}(\mathbb{R}^2)}:=\left(\|u\|^2_{L^2(\mathbb{R}^2)}+|u|^2_{J^\alpha_{L,0\pi}(\mathbb{R}^2)}\right)^{1/2},\\
	&\|u\|_{J^\alpha_{L,\pi/2}(\mathbb{R}^2)}:=\left (\|u\|^2_{L^2(\mathbb{R}^2)}+|u|^2_{J^\alpha_{L,\pi/2}(\mathbb{R}^2)}\right)^{1/2},\\
	&\|u\|_{J^\alpha_{L,\pi}(\mathbb{R}^2)}:=\left (\|u\|^2_{L^2(\mathbb{R}^2)}+|u|^2_{J^\alpha_{L,\pi}(\mathbb{R}^2)}\right )^{1/2},\\
	&\|u\|_{J^\alpha_{L,3\pi/2}(\mathbb{R}^2)}:=\left (\|u\|^2_{L^2(\mathbb{R}^2)}+|u|^2_{J^\alpha_{L,3\pi/2}(\mathbb{R}^2)}\right)^{1/2},
	\end{aligned}
	\end{equation*}
	and let $ J^\alpha_{L,0\pi}(\mathbb{R}^2) $, $ J^\alpha_{L,\pi/2}(\mathbb{R}^2) $, $ J^\alpha_{L,\pi}(\mathbb{R}^2) $, and $ J^\alpha_{L,3\pi/2}(\mathbb{R}^2) $ denote the closures of $C^\infty_0(\mathbb{R}^2)$ with respect to $\|\cdot\|_{J^\alpha_{L,0\pi}(\mathbb{R}^2)}$, 	$\|\cdot\|_{J^\alpha_{L,\pi/2}(\mathbb{R}^2)}$, 	$\|\cdot\|_{J^\alpha_{L,\pi}(\mathbb{R}^2)}$, and $\|\cdot\|_{J^\alpha_{L,3\pi/2}(\mathbb{R}^2)}$, respectively.
\end{definition}

\begin{definition}[\cite{Ervin2007}]
	Let $\Omega \subset \mathbb{R}^2$.	Define the spaces $ J^\alpha_{L,0\pi,0}(\Omega) $, $ J^\alpha_{L,\pi/2,0}(\Omega) $, $ J^\alpha_{L,\pi,0}(\Omega) $, $ J^\alpha_{L,3\pi/2,0}(\Omega) $ as the closures of $C^\infty_0(\Omega)$ under their respective norms.
\end{definition}

\begin{lemma}[\cite{Ervin2007}]\label{lemequnorm2D}
	Let $\Omega \subset \mathbb{R}^2$, $ \alpha>0 $, and $ \alpha\neq n-1/2 $, $n\in\mathbb{N}$ be given. The spaces $ J^\alpha_{L,0\pi,0}(\Omega) $, $ J^\alpha_{L,\pi/2,0}(\Omega) $, $ J^\alpha_{L,\pi,0}(\Omega) $, $ J^\alpha_{L,3\pi/2,0}(\Omega) $, and $H^\alpha_0(\Omega)$ are equal  with equivalent semi-norms and norms. Similarly for $\Omega \subset \mathbb{R}$, we have that  $J^{\alpha}_{L,0}(\Omega)$, $J^{\alpha}_{R,0}(\Omega)$, and $H^\alpha_0(\Omega)$ are equal with equivalent semi-norms and norms.
\end{lemma}

\begin{lemma}[\cite{Ervin2007}]
	Let $u\in J^\alpha_{L,\theta,0}(\Omega)$, $\theta\in\{0\pi,\pi/2,\pi,3\pi/2\}$, $\Omega \subset \mathbb{R}^2$. There are
	\begin{equation*}
	\|u\|_{L^2(\Omega)}\leq C|u|_{J^\alpha_{L,\theta,0}(\Omega)}
	\end{equation*}
	and for $ 0<\nu<\alpha$,
	\begin{equation*}
	|u|_{J^\nu_{L,\theta,0}(\Omega)}<C|u|_{J^\alpha_{L,\theta,0}(\Omega)}.
	\end{equation*}
\end{lemma}

\begin{lemma}[\cite{Ervin2006}]
	\label{lemintinversede}	The left (right) Riemann-Liouville fractional derivative of order $\alpha$ acts as a left inverse of the left (right) Riemann-Liouville fractional integral of order $ \alpha $, i.e.,
	\begin{align*}
	&~_aD^\alpha_x~_aD^{-\alpha}_xu(x)=u(x),\\
	&~_xD^\alpha_b~_xD^{-\alpha}_bu(x)=u(x),\;\forall~ \alpha>0.
	\end{align*}
\end{lemma}

\begin{lemma}[\cite{Ervin2006}]\label{lemLRcosnorm}
	Let $ \alpha>0 $ and $ n $ be the smallest integer greater than $ \alpha $ $( n-1\leq \alpha<n)$. Then for a real valued function $u(x)$, there exists
	\begin{equation*}
	(~_{-\infty} D^\alpha_xu,~_x D^\alpha_\infty u)=\cos(\pi\alpha)\|~_{-\infty} D^\alpha_xu\|^{2}_{L^2(\mathbb{R})}.
	\end{equation*}
\end{lemma}

\begin{lemma}[\cite{Li2009}]\label{lemdisjoint}
	For all $0<\alpha<1$, $\Omega=(a,b)\subset\mathbb{R}$, if $u\in~_0H^1(\Omega)$ and $v\in~_0H^{\frac{\alpha}{2}}(\Omega)$, then
	\begin{equation*}
	\left (~_aD^\alpha_xu,v\right )_{\Omega}=\left(~_aD^{\alpha/2}_xu,~_xD^{\alpha/2}_bv\right)_\Omega.
	\end{equation*}
\end{lemma}


\begin{lemma}\label{lemmalast}
	Suppose $0<\alpha<1$, $\Omega=(a,b)$ be a bounded open subinterval of $\mathbb{R}$, $ u\in ~_0H^1(\Omega)$, $v\in~_0H^\alpha(\Omega)$, and ${\rm supp}(v)\subset I\subset\Omega$. Then there exists
	\begin{equation*}
	(~_aD^\alpha_xu,v)_I=(~_aD^{\alpha/2}_xu,~_xD^{\alpha/2}_bv)_\Omega.
	\end{equation*}
\end{lemma}
\begin{proof}
	According to ${\rm supp}(v)\subset I\subset\Omega$ and Lemma \ref{lemdisjoint}, we have
	\begin{equation*}
	(~_aD^\alpha_xu,v)_I=(~_aD^\alpha_xu,v)_\Omega=(~_aD^{\alpha/2}_xu,~_xD^{\alpha/2}_bv)_\Omega.
	\end{equation*}
\end{proof}
\begin{lemma}[\cite{Kilbas2006}]\label{lemCaputo0}
	Suppose $\alpha>0$ and $n-1\leq\alpha<n$, $n\in \mathbb{N}$. Let $u\in C^n[a,b]$. Then
	\begin{equation*} 
	~_aI^{n-\alpha}_xD^nu\in C[a,b],\quad ~_xI^{n-\alpha}_bD^nu\in C[a,b],
	\end{equation*}
	where $D=d/dx$. Furthermore if $\alpha\notin \mathbb{N}$, one has
	\begin{equation*}
	~_aI^{n-\alpha}_xD^nu(a)=~_aI^{n-\alpha}_xD^nu(b)=0.
	\end{equation*}
\end{lemma}
\section{Central LDG scheme for fractional diffusion equation} 
\label{}
Now, we design the spatial semi-discrete scheme by using the central LDG method to discretize the space fractional derivatives.

Firstly, we define the inner product on the simplex $I$ and  over the surface of $I$  for two continuous functions $f, g$,
\begin{equation*}
(f,g)_I=\int_{I}f(\mathbf{x})g(\mathbf{x})d\mathbf{x},\ (f,g)_{\partial
	I}=\int_{\partial I}f(\mathbf{x})g(\mathbf{x})ds.
\end{equation*}
Following the standard approach for the development of the LDG method for equations with higher order derivatives in \cite{Qiu2015}, we introduce the auxiliary variable $\mathbf{\bar{q}}=(q^{x}_L+q^x_R,q^y_L+q^y_R)$, and assume $u\in C^1(\Omega)$, $~_0D^{\alpha-2}_xu$, $~_xD^{\alpha-2}_1u$, $~_0D^{\beta-2}_yu$, and $~_yD^{\beta-2}_1u\in H^1(\Omega)$. Moreover, according to Lemmas \ref{lemintinversede} and \ref{lemCaputo0}, Eq.  (\ref{equation2D}) can be written as
\begin{equation}\label{eqldlsty1}
\left\{
\begin{aligned}
&\frac{\partial u}{\partial t}=\nabla\cdot \mathbf{\bar{q}}+f~~~~~(x,y,t)\in \Omega\times[0,T],\\
&~_0D^{2-\alpha}_xq^x_L=\frac{\partial}{\partial x}u~~~~(x,y,t)\in \Omega\times[0,T],\\
&~_xD^{2-\alpha}_1q^x_R=\frac{\partial}{\partial x}u~~~~(x,y,t)\in \Omega\times[0,T],\\
&~_0D^{2-\beta}_yq^y_L=\frac{\partial}{\partial y}u~~~~~(x,y,t)\in \Omega\times[0,T],\\
&~_yD^{2-\beta}_1q^y_R=\frac{\partial}{\partial y}u~~~~~(x,y,t)\in \Omega\times[0,T],\\
&u(x,y,0)=g(x,y)~~~(x,y)\in \Omega,\\
&u(x,y,t)=0~~~~~~~~~~(x,y,t)\in (\mathbb{R}^2\backslash\Omega)\times[0,T],\\
&q^x_L(0,y,t)=q^x_R(1,y,t)=q^y_L(x,0,t)=q^y_R(x,1,t)=0,\ t\in[0,T].
\end{aligned}
\right.
\end{equation}
Further define two operators $\mathbf{D}^{\alpha,\beta}_L\mathbf{q}=(~_0D^\alpha_xq^x,~_0D^\beta_yq^y)$ and $\mathbf{D}^{\alpha,\beta}_R\mathbf{q}=(~_xD^\alpha_1q^x,~_yD^\beta_1q^y)$, where $\mathbf{q}=(q^x,q^y)$. Denote $ \mathbf{q}_L=(q^x_L,q^y_L) $ and $ \mathbf{q}_R=(q^x_R, q^y_R) $.  Then Eq. \eqref{eqldlsty1} can be recast as
\begin{equation*}
\left \{
\begin{aligned}
&\frac{\partial u}{\partial t}=\nabla\cdot (\mathbf{q}_L+\mathbf{q}_R)+f~~~(x,y,t)\in \Omega\times[0,T],\\
&\mathbf{D}^{2-\alpha,2-\beta}_L\mathbf{q}_L=\nabla u~~~~~~~~~~~(x,y,t)\in \Omega\times[0,T],\\
&\mathbf{D}^{2-\alpha,2-\beta}_R\mathbf{q}_R=\nabla u~~~~~~~~~~~(x,y,t)\in \Omega\times[0,T],\\
&u(x,y,0)=g(x,y)~~~~~~~~~~~(x,y)\in \Omega,\\
&u(x,y,t)=0~~~~~~~~~~~~~~~~~~(x,y,t)\in (\mathbb{R}^2\backslash\Omega)\times[0,T],\\
&q^x_L(0,y,t)=q^x_R(1,y,t)=q^y_L(x,0,t)=q^y_R(x,1,t)=0\ ~~~~t\in[0,T].
\end{aligned}
\right.
\end{equation*}
Next, we introduce some notations to get the spatial semi-discrete scheme. Let the mesh size $h=1/N$, $N\in\mathbb{N}$, $\{x_j\}$, $\{y_j\}$ be a partition of $[0,1]$ and grid points $x_j=y_j=jh$\, $(j=0,1,\cdots,N)$. Denote $x_{j+\frac{1}{2}}=\frac{1}{2}(x_{j+1}+x_j)$ and  $y_{j+\frac{1}{2}}=\frac{1}{2}(y_{j+1}+y_j)$; $I_{i,j}=(x_{i},x_{i+1})\times(y_{j},y_{j+1})$ with  $(i,j=0,1,\cdots,N-1)$ and $\bar{I}_{i,j}=(x_{i-\frac{1}{2}},x_{i+\frac{1}{2}})\times(y_{j-\frac{1}{2}},y_{j+\frac{1}{2}})$ with $(i,j=0,1,\cdots,N)$ are two sets of the overlapping cells and we set $x_{-\frac{1}{2}}=x_0$, $y_{-\frac{1}{2}}=y_0$, $x_{N+\frac{1}{2}}=x_N$, and $y_{N+\frac{1}{2}}=y_N$ specially. Denote $V_h$ as the set of piecewise polynomials of degree $k$ over the subintervals $\{I_{i,j}\}$ with no continuity assumed across the subinterval boundaries, and the set of the ones over the subintervals $\{\bar{I}_{i,j}\}$ is denoted as $W_h$.

To get the central LDG scheme by using  the above two discrete function spaces $V_h$ and $W_h$, let $u\in H^{1}(0,T,H^1_0(\Omega)) $ and $\mathbf{q}_L, \mathbf{q}_R\in L^2(0,T,H^{(2-\alpha)/2}(\Omega))\times L^2(0,T,H^{(2-\beta)/2}(\Omega))$. Furthermore, defining $u_{1,h}$, $\mathbf{q}_{L,1,h}$, $\mathbf{q}_{R,1,h}$ as the approximations of  $u$, $\mathbf{q}_{L}$, and $\mathbf{q}_{R}$ on mesh $\{I_{i,j}\}$, respectively, and $u_{2,h}$, $\mathbf{q}_{L,2,h}$, $\mathbf{q}_{R,2,h}$ as the approximations of  $u$, $\mathbf{q}_{L}$, and $\mathbf{q}_{R}$ on mesh $\{\bar{I}_{i,j}\}$, respectively, then we get the central LDG scheme which involves two pieces of approximate solutions defined on overlapping cells, i.e., find $u_{1,h}\in H^1(0,T,V_h)$, $u_{2,h}\in H^1(0,T,W_h)$, $\mathbf{q}_{L,1,h}=(q^x_{L,1,h}, q^y_{L,1,h})\in (L^2(0,T,V_h))^2$,  $\mathbf{q}_{L,2,h}=(q^x_{L,2,h}, q^y_{L,2,h})\in (L^2(0,T,W_h))^2$, and $\mathbf{q}_{R,1,h}=(q^x_{R,1,h}, q^y_{R,1,h})\in (L^2(0,T,V_h))^2$, $\mathbf{q}_{R,2,h}=(q^x_{R,2,h}, q^y_{R,2,h})\in (L^2(0,T,W_h))^2$  satisfying
\begin{equation}\label{equationdis1}
\left\{
\begin{aligned}
&\left (\frac{\partial u_{1,h}}{\partial t},v_1\right )_{I_{i,j}}=\frac{1}{\tau_{max}}(u_{2,h}-u_{1,h},v_1)_{I_{i,j}}-(\mathbf{q}_{L,2,h}+\mathbf{q}_{R,2,h},\nabla v_1)_{I_{i,j}}\\
&\qquad\qquad\qquad\qquad\qquad\qquad+(\mathbf{n}\cdot(\mathbf{q}_{L,2,h}+\mathbf{q}_{R,2,h}),v_1)_{\partial I_{i,j}}+(f,v_1)_{I_{i,j}},\\
&\left (\mathbf{D}^{(2-\alpha)/2,(2-\beta)/2}_L\mathbf{q}_{L,2,h},\mathbf{D}^{(2-\alpha)/2,(2-\beta)/2}_R\boldsymbol{\phi}^{i,j}_{L,2}\right )_{\Omega}\\
&\qquad\qquad\qquad\qquad\qquad\qquad=-(u_{1,h},\nabla\boldsymbol{\phi}^{i,j}_{L,2})_{\bar{I}_{i,j}}+(u_{1,h},\mathbf{n}\cdot\boldsymbol{\phi}^{i,j}_{L,2})_{\partial \bar{I}_{i,j}},\\
&\left (\mathbf{D}^{(2-\alpha)/2,(2-\beta)/2}_R\mathbf{q}_{R,2,h},\mathbf{D}^{(2-\alpha)/2,(2-\beta)/2}_L\boldsymbol{\phi}^{i,j}_{R,2}\right )_{\Omega}\\
&\qquad\qquad\qquad\qquad\qquad\qquad=-(u_{1,h},\nabla\boldsymbol{\phi}^{i,j}_{R,2})_{\bar{I}_{i,j}}+(u_{1,h},\mathbf{n}\cdot\boldsymbol{\phi}^{i,j}_{R,2})_{\partial \bar{I}_{i,j}},\\
&u_{1,h}(x,y,0)=g(x,y)\quad\quad\quad\quad\quad\quad\qquad\quad~~(x,y)\in \Omega,\\
&u_{1,h}(x,y,t)=0\quad\quad\quad\quad\quad\quad\quad\qquad\qquad\quad(x,y,t)\in (\mathbb{R}^2\backslash\Omega)\times[0,T],\\
&q^x_{L,2,h}(0,y,t)=q^x_{R,2,h}(1,y,t)=0 \quad\quad\quad\quad\quad t\in[0,T],\\
&q^y_{L,2,h}(x,0,t)=q^y_{R,2,h}(x,1,t)=0\quad\quad\qquad\quad t\in[0,T],
\end{aligned}
\right.
\end{equation}
and
\begin{equation}\label{equationdis2}
\left\{
\begin{aligned}
&\left (\frac{\partial u_{2,h}}{\partial t},v_2\right )_{\bar{I}_{i,j}}=\frac{1}{\tau_{max}}(u_{1,h}-u_{2,h},v_2)_{\bar{I}_{i,j}}-(\mathbf{q}_{L,1,h}+\mathbf{q}_{R,1,h},\nabla v_2)_{\bar{I}_{i,j}}\\
&\qquad\qquad\qquad\qquad\qquad\qquad+(\mathbf{n}\cdot(\mathbf{q}_{L,1,h}+\mathbf{q}_{R,1,h}),v_2)_{\partial \bar{I}_{i,j}}+(f,v_2)_{\bar{I}_{i,j}},\\
&\left (\mathbf{D}^{(2-\alpha)/2,(2-\beta)/2}_L\mathbf{q}_{L,1,h},\mathbf{D}^{(2-\alpha)/2,(2-\beta)/2}_R\boldsymbol{\phi}^{i,j}_{L,1}\right )_{\Omega}\\
&\qquad\qquad\qquad\qquad\qquad\qquad=-(u_{2,h},\nabla\boldsymbol{\phi}^{i,j}_{L,1})_{I_{i,j}}+(u_{2,h},\mathbf{n}\cdot\boldsymbol{\phi}^{i,j}_{L,1})_{\partial I_{i,j}},\\
&\left (\mathbf{D}^{(2-\alpha)/2,(2-\beta)/2}_R\mathbf{q}_{R,1,h},\mathbf{D}^{(2-\alpha)/2,(2-\beta)/2}_L\boldsymbol{\phi}^{i,j}_{R,1}\right )_{\Omega}\\
&\qquad\qquad\qquad\qquad\qquad\qquad=-(u_{2,h},\nabla\boldsymbol{\phi}^{i,j}_{R,1})_{I_{i,j}}+(u_{2,h},\mathbf{n}\cdot\boldsymbol{\phi}^{i,j}_{R,1})_{\partial I_{i,j}},\\
&u_{2,h}(x,y,0)=g(x,y)\quad\quad\quad\quad\quad\quad\qquad\qquad(x,y)\in \Omega,\\
&u_{2,h}(x,y,t)=0\quad\quad\quad\quad\quad\quad\quad\quad\qquad\qquad~(x,y,t)\in (\mathbb{R}^2\backslash\Omega)\times[0,T],\\
&q^x_{L,1,h}(0,y,t)=q^x_{R,1,h}(1,y,t)=0 \quad\quad\qquad\qquad t\in[0,T],\\
&q^y_{L,1,h}(x,0,t)=q^y_{R,1,h}(x,1,t)=0\quad\quad\qquad\qquad t\in[0,T],
\end{aligned}
\right.
\end{equation}
for any function $v_1\in V_h$, $v_2\in W_h$, $\boldsymbol{\phi}^{i,j}_{L,1}\in(V_h)^2$, $\boldsymbol{\phi}^{i,j}_{L,2}\in(W_h)^2$, $\boldsymbol{\phi}^{i,j}_{R,1}\in(V_h)^2$, $\boldsymbol{\phi}^{i,j}_{R,2}\in(W_h)^2$ with ${\rm supp}(\boldsymbol{\phi}^{i,j}_{L,1})$ and ${\rm supp}(\boldsymbol{\phi}^{i,j}_{R,1})\subset I_{i,j}$, and $ {\rm supp}(\boldsymbol{\phi}^{i,j}_{L,2})$ and ${\rm supp}(\boldsymbol{\phi}^{i,j}_{R,2})\subset \bar{I}_{i,j}$. Here, $\mathbf{n}$ stands for the unit outer normal of the corresponding cell, and $\tau_{max}$ is an upper bound for the time step size due to the CFL restriction and Lemma \ref{lemmalast}. For the simplicity of the theoretical analysis, we denote
\begin{equation} 
\begin{aligned}
&B_1(u_{1,h},u_{2,h},\mathbf{q}_{L,2,h},\mathbf{q}_{R,2,h},v_1,\boldsymbol{\phi}_{L,2},\boldsymbol{\phi}_{R,2})\\
=&\sum_{i=0,j=0}^{i=N-1,j=N-1}\int_{0}^{T}\left(\frac{\partial u_{1,h}}{\partial t},v_1\right)_{I_{i,j}}+\frac{1}{\tau_{max}}(u_{1,h}-u_{2,h},v_1)_{I_{i,j}}\\
&\qquad\qquad\qquad\quad+(\mathbf{q}_{L,2,h}+\mathbf{q}_{R,2,h},\nabla v_1)_{I_{i,j}}
-(\mathbf{n}\cdot(\mathbf{q}_{L,2,h}+\mathbf{q}_{R,2,h}),v_1)_{\partial I_{i,j}}dt\\
&+\sum_{i=0,j=0}^{i=N,j=N}\int_{0}^{T}\left (\mathbf{D}^{(2-\alpha)/2,(2-\beta)/2}_L\mathbf{q}_{L,2,h},\mathbf{D}^{(2-\alpha)/2,(2-\beta)/2}_R\boldsymbol{\phi}^{i,j}_{L,2}\right )_{\Omega}\\
&\qquad\qquad\qquad+(u_{1,h},\nabla\boldsymbol{\phi}^{i,j}_{L,2})_{\bar{I}_{i,j}}-(u_{1,h},\mathbf{n}\cdot\boldsymbol{\phi}^{i,j}_{L,2})_{\partial \bar{I}_{i,j}}\\
&\qquad\qquad\qquad+\left (\mathbf{D}^{(2-\alpha)/2,(2-\beta)/2}_R\mathbf{q}_{R,2,h},\mathbf{D}^{(2-\alpha)/2,(2-\beta)/2}_L\boldsymbol{\phi}^{i,j}_{R,2}\right )_{\Omega}\\&\qquad\qquad\qquad+(u_{1,h},\nabla\boldsymbol{\phi}^{i,j}_{R,2})_{\bar{I}_{i,j}}-(u_{1,h},\mathbf{n}\cdot\boldsymbol{\phi}^{i,j}_{R,2})_{\partial \bar{I}_{i,j}}dt\\
\end{aligned}
\end{equation}
and
\begin{equation}
\begin{aligned}
&B_2(u_{2,h},u_{1,h},\mathbf{q}_{L,1,h},\mathbf{q}_{R,1,h},v_2,\boldsymbol{\phi}_{L,1},\boldsymbol{\phi}_{R,1})\\
=&\sum_{i=0,j=0}^{i=N,j=N}\int_{0}^{T}\left(\frac{\partial u_{2,h}}{\partial t},v_2\right)_{\bar{I}_{i,j}}+\frac{1}{\tau_{max}}(u_{2,h}-u_{1,h},v_2)_{\bar{I}_{i,j}}\\
&\qquad\qquad\quad+(\mathbf{q}_{L,1,h}+\mathbf{q}_{R,1,h},\nabla v_2)_{\bar{I}_{i,j}}
-(\mathbf{n}\cdot(\mathbf{q}_{L,1,h}+\mathbf{q}_{R,1,h}),v_2)_{\partial \bar{I}_{i,j}}dt\\
&+\sum_{i=0,j=0}^{i=N-1,j=N-1}\int_{0}^{T}\left (\mathbf{D}^{(2-\alpha)/2,(2-\beta)/2}_L\mathbf{q}_{L,1,h},\mathbf{D}^{(2-\alpha)/2,(2-\beta)/2}_R\boldsymbol{\phi}^{i,j}_{L,1}\right )_{\Omega}\\
&\qquad\qquad\qquad\qquad+(u_{2,h},\nabla\boldsymbol{\phi}^{i,j}_{L,1})_{I_{i,j}}-(u_{2,h},\mathbf{n}\cdot\boldsymbol{\phi}^{i,j}_{L,1})_{\partial I_{i,j}}\\
&\qquad\qquad\qquad\qquad+\left (\mathbf{D}^{(2-\alpha)/2,(2-\beta)/2}_R\mathbf{q}_{R,1,h},\mathbf{D}^{(2-\alpha)/2,(2-\beta)/2}_L\boldsymbol{\phi}^{i,j}_{R,1}\right )_{\Omega}\\
&\qquad\qquad\qquad\qquad+(u_{2,h},\nabla\boldsymbol{\phi}^{i,j}_{R,1})_{I_{i,j}}-(u_{2,h},\mathbf{n}\cdot\boldsymbol{\phi}^{i,j}_{R,1})_{\partial I_{i,j}} dt,
\end{aligned}
\end{equation}
where $\boldsymbol{\phi}_{L,1}=\{\boldsymbol{\phi}^{i,j}_{L,1}\}^{i=N-1,j=N-1}_{i=0,j=0}$,  $\boldsymbol{\phi}_{R,1}=\{\boldsymbol{\phi}^{i,j}_{R,1}\}^{i=N-1,j=N-1}_{i=0,j=0}$,  $\boldsymbol{\phi}_{L,2}=\{\boldsymbol{\phi}^{i,j}_{L,2}\}^{i=N,j=N}_{i=0,j=0}$, and  $\boldsymbol{\phi}_{R,2}=\{\boldsymbol{\phi}^{i,j}_{R,2}\}^{i=N,j=N}_{i=0,j=0}$. Therefore, the schemes \eqref{equationdis1} and \eqref{equationdis2} can be transformed as: find $u_{1,h}\in H^1(0,T,V_h)$, $u_{2,h}\in H^1(0,T,W_h)$, $\mathbf{q}_{L,1,h}=(q^x_{L,1,h}, q^y_{L,1,h})\in (L^2(0,T,V_h))^2$,  $\mathbf{q}_{L,2,h}=(q^x_{L,2,h}, q^y_{L,2,h})\in (L^2(0,T,W_h))^2$, $\mathbf{q}_{R,1,h}=(q^x_{R,1,h}, q^y_{R,1,h})\in (L^2(0,T,V_h))^2$, and $\mathbf{q}_{R,2,h}=(q^x_{R,2,h}, q^y_{R,2,h})\in (L^2(0,T,W_h))^2$ satisfying
\begin{equation}\label{equBin1sty}
B_1(u_{1,h},u_{2,h},\mathbf{q}_{L,2,h},\mathbf{q}_{R,2,h},v_1,\boldsymbol{\phi}_{L,2},\boldsymbol{\phi}_{R,2})=\int_{0}^{T}(f,v_1)dt
\end{equation}
and
\begin{equation}\label{equBin2sty}
B_2(u_{2,h},u_{1,h},\mathbf{q}_{L,1,h},\mathbf{q}_{R,1,h},v_2,\boldsymbol{\phi}_{L,1},\boldsymbol{\phi}_{R,1})=\int_{0}^{T}(f,v_2)dt
\end{equation}
for any function $v_1\in V_h$, $v_2\in V_h$, $\boldsymbol{\phi}^{i,j}_{L,1}\in(V_h)^2$, $\boldsymbol{\phi}^{i,j}_{L,2}\in(W_h)^2$, $\boldsymbol{\phi}^{i,j}_{R,1}\in(V_h)^2$, $\boldsymbol{\phi}^{i,j}_{R,2}\in(W_h)^2$ with ${\rm supp}(\boldsymbol{\phi}^{i,j}_{L,1})$ and ${\rm supp}( \boldsymbol{\phi}^{i,j}_{R,1})\subset I_{i,j}$, and ${\rm supp}( \boldsymbol{\phi}^{i,j}_{L,2})$ and ${\rm supp}( \boldsymbol{\phi}^{i,j}_{R,2})\subset \bar{I}_{i,j}$.
\section{Stability analysis and error estimates}
Let $(\tilde{u}_{1,h},\tilde{\mathbf{q}}_{L,1,h},\tilde{\mathbf{q}}_{R,1,h})$ be the approximation of the solution $(u_{1,h},\mathbf{q}_{L,1,h},\mathbf{q}_{R,1,h})$ in $H^1(0,T,V_h)\times (L^2(0,T,V_h))^2\times (L^2(0,T,V_h))^2$, and $(\tilde{u}_{2,h},\tilde{\mathbf{q}}_{L,2,h},\tilde{\mathbf{q}}_{R,2,h})$ be the approximation of $(u_{2,h},\mathbf{q}_{L,2,h},\mathbf{q}_{R,2,h})$ in $H^1(0,T,W_h) \times (L^2(0,T,W_h))^2\times (L^2(0,T,W_h))^2$.
Denote $e_{u_{1,h}}=u_{1,h}-\tilde{u}_{1,h}$,  $e_{\mathbf{q}_{L,1,h}}=\mathbf{q}_{L,1,h}-\tilde{\mathbf{q}}_{L,1,h}$,  $e_{\mathbf{q}_{R,1,h}}=\mathbf{q}_{R,1,h}-\tilde{\mathbf{q}}_{R,1,h}$, $e_{u_{2,h}}=u_{2,h}-\tilde{u}_{2,h}$,  $e_{\mathbf{q}_{L,2,h}}=\mathbf{q}_{L,2,h}-\tilde{\mathbf{q}}_{L,2,h}$, and  $e_{\mathbf{q}_{R,2,h}}=\mathbf{q}_{R,2,h}-\tilde{\mathbf{q}}_{R,2,h}$ as the round-off errors.
\subsection{Stability analysis}
\begin{theorem}\label{thml2stab}
	Numerical schemes (\ref{equBin1sty}) and (\ref{equBin2sty}) are $L^2$ stable, and for all $t\in(0,T)$ we have
	\begin{equation}\label{equressta}
	\begin{aligned}
	\|e_{u_{1,h}}(T)\|^2_{L^{2}(\Omega)}+\|e_{u_{2,h}}(T)\|^2_{L^{2}(\Omega)}\leq\|e_{u_{1,h}}(0)\|^2_{L^{2}(\Omega)}+\|e_{u_{2,h}}(0)\|^2_{L^{2}(\Omega)}.
	\end{aligned}
	\end{equation}
\end{theorem}
\begin{proof}
	According to the definitions of $ e_{u_{1,h}} $, $ e_{\mathbf{q}_{L,2,h}} $, $ e_{\mathbf{q}_{R,2,h}} $, and Eqs. \eqref{equBin1sty} and \eqref{equBin2sty}, we obtain
	\begin{equation}\label{equBin1eq0}
	\begin{aligned}
	&B_{1}(e_{u_{1,h}},e_{u_{2,h}},e_{\mathbf{q}_{L,2,h}},e_{\mathbf{q}_{R,2,h}},v_1,\boldsymbol{\phi}_{L,2},\boldsymbol{\phi}_{R,2})=0,\\
	&B_{2}(e_{u_{2,h}},e_{u_{1,h}},e_{\mathbf{q}_{L,1,h}},e_{\mathbf{q}_{R,1,h}},v_2,\boldsymbol{\phi}_{L,1},\boldsymbol{\phi}_{R,1})=0
	\end{aligned}
	\end{equation}
for any function $v_1\in V_h$, $v_2\in V_h$, $\boldsymbol{\phi}^{i,j}_{L,1}\in(V_h)^2$, $\boldsymbol{\phi}^{i,j}_{L,2}\in(W_h)^2$, $\boldsymbol{\phi}^{i,j}_{R,1}\in(V_h)^2$, $\boldsymbol{\phi}^{i,j}_{R,2}\in(W_h)^2$ with ${\rm supp}(\boldsymbol{\phi}^{i,j}_{L,1})$ and ${\rm supp}( \boldsymbol{\phi}^{i,j}_{R,1})\subset I_{i,j}$, and ${\rm supp}( \boldsymbol{\phi}^{i,j}_{L,2})$ and ${\rm supp}( \boldsymbol{\phi}^{i,j}_{R,2})\subset \bar{I}_{i,j}$.
	Denoting $\chi_{\bar{I}_{i,j}}$ as the characteristic function of the closure of the cell $\bar{I}_{i,j}$
	and taking $v_1=e_{u_{1,h}} $, $ \boldsymbol{\phi}^{i,j}_{L,2}=e_{\mathbf{q}_{L,2,h}}\chi_{\bar{I}_{i,j}} $, and $ \boldsymbol{\phi}^{i,j}_{R,2}=e_{\mathbf{q}_{R,2,h}}\chi_{\bar{I}_{i,j}}$ for Eq. \eqref{equBin1sty},  we get
	\begin{equation*}
	\begin{aligned}
	&B_{1}(e_{u_{1,h}},e_{u_{2,h}},e_{\mathbf{q}_{L,2,h}},e_{\mathbf{q}_{R,2,h}},e_{u_{1,h}},\boldsymbol{\phi}_{L,2},\boldsymbol{\phi}_{R,2})\\
	=&\sum_{i=0,j=0}^{i=N-1,j=N-1}\int_{0}^{T}\left(\frac{\partial e_{u_{1,h}}}{\partial t},e_{u_{1,h}}\right)_{I_{i,j}}+\frac{1}{\tau_{max}}(e_{u_{1,h}}-e_{u_{2,h}},e_{u_{1,h}})_{I_{i,j}}\\
	&\qquad\qquad\quad+(e_{\mathbf{q}_{L,2,h}}+e_{\mathbf{q}_{R,2,h}},\nabla e_{u_{1,h}})_{I_{i,j}}
	-(\mathbf{n}\cdot(e_{\mathbf{q}_{L,2,h}}+e_{\mathbf{q}_{R,2,h}}),e_{u_{1,h}})_{\partial I_{i,j}}dt\\
	&+\sum_{i=0,j=0}^{i=N,j=N}\int_{0}^{T}\left (\mathbf{D}^{(2-\alpha)/2,(2-\beta)/2}_L e_{\mathbf{q}_{L,2,h}},\mathbf{D}^{(2-\alpha)/2,(2-\beta)/2}_R (e_{\mathbf{q}_{L,2,h}}\chi_{\bar{I}_{i,j}})\right )_{\Omega}\\
	&\qquad\qquad\qquad+\left(e_{u_{1,h}},\nabla (e_{\mathbf{q}_{L,2,h}}\chi_{\bar{I}_{i,j}})\right)_{\bar{I}_{i,j}}-\left(e_{u_{1,h}},\mathbf{n}\cdot (e_{\mathbf{q}_{L,2,h}}\chi_{\bar{I}_{i,j}})\right)_{\partial \bar{I}_{i,j}}\\
	&\qquad\qquad\qquad+\left (\mathbf{D}^{(2-\alpha)/2,(2-\beta)/2}_R e_{\mathbf{q}_{R,2,h}},\mathbf{D}^{(2-\alpha)/2,(2-\beta)/2}_L (e_{\mathbf{q}_{R,2,h}}\chi_{\bar{I}_{i,j}})\right )_{\Omega}\\
	&\qquad\qquad\qquad+\left(e_{u_{1,h}},\nabla (e_{\mathbf{q}_{R,2,h}}\chi_{\bar{I}_{i,j}})\right)_{\bar{I}_{i,j}}-\left(e_{u_{1,h}},\mathbf{n}\cdot (e_{\mathbf{q}_{R,2,h}}\chi_{\bar{I}_{i,j}})\right)_{\partial \bar{I}_{i,j}}dt.\\
	\end{aligned}
	\end{equation*}
Using Green's theorem leads to
	\begin{equation*}
	\begin{aligned}
	&B_{1}(e_{u_{1,h}},e_{u_{2,h}},e_{\mathbf{q}_{L,2,h}},e_{\mathbf{q}_{R,2,h}},e_{u_{1,h}},\boldsymbol{\phi}_{L,2},\boldsymbol{\phi}_{R,2})\\
	=&\sum_{i=0,j=0}^{i=N-1,j=N-1}\int_{0}^{T}\left(\frac{\partial e_{u_{1,h}}}{\partial t},e_{u_{1,h}}\right)_{I_{i,j}}+\frac{1}{\tau_{max}}(e_{u_{1,h}}-e_{u_{2,h}},e_{u_{1,h}})_{I_{i,j}}dt\\
	&+\int_{0}^{T}\left (\mathbf{D}^{(2-\alpha)/2,(2-\beta)/2}_L e_{\mathbf{q}_{L,2,h}},\mathbf{D}^{(2-\alpha)/2,(2-\beta)/2}_R e_{\mathbf{q}_{L,2,h}}\right )_{\Omega}\\
	&\qquad\quad+\left (\mathbf{D}^{(2-\alpha)/2,(2-\beta)/2}_R e_{\mathbf{q}_{R,2,h}},\mathbf{D}^{(2-\alpha)/2,(2-\beta)/2}_L e_{\mathbf{q}_{R,2,h}}\right )_{\Omega}dt.
	\end{aligned}
	\end{equation*}
	Combining Lemma \ref{lemLRcosnorm}, we obtain
	\begin{equation*}
	\begin{aligned}
	&B_{1}(e_{u_{1,h}},e_{u_{2,h}},e_{\mathbf{q}_{L,2,h}},e_{\mathbf{q}_{R,2,h}},e_{u_{1,h}},\boldsymbol{\phi}_{L,2},\boldsymbol{\phi}_{R,2})\\
	=&\frac{1}{2}\int_{0}^{T}\frac{\partial}{\partial t}\|e_{u_{1,h}}\|^2_{L^2(\Omega)}dt+\frac{1}{\tau_{max}}\int_{0}^{T}(e_{u_{1,h}}-e_{u_{2,h}},e_{u_{1,h}})_{\Omega}dt\\
	&+\int_{0}^{T}\left (\int_0^1\cos\left (\frac{2-\alpha}{2}\pi\right )|e_{q^x_{L,2,h}}(\cdot,y)|^2_{H^{(2-\alpha)/2}(0,1)}dy\right .\\
	&\qquad\quad+\int_0^1\cos\left (\frac{2-\beta}{2}\pi\right )|e_{q^y_{L,2,h}}(x,\cdot)|^2_{H^{(2-\beta)/2}(0,1)}dx\\
	&\qquad\quad+\int_0^1\cos\left (\frac{2-\alpha}{2}\pi\right )|e_{q^x_{R,2,h}}(\cdot,y)|^2_{H^{(2-\alpha)/2}(0,1)}dy\\
	&\qquad\quad\left .+\int_0^1\cos\left (\frac{2-\beta}{2}\pi\right )|e_{q^y_{R,2,h}}(x,\cdot)|^2_{H^{(2-\beta)/2}(0,1)}dx\right )dt.
	\end{aligned}
	\end{equation*}
Similarly, there also exists 
	\begin{equation*}
	\begin{aligned}
	&B_{2}(e_{u_{2,h}},e_{u_{1,h}},e_{\mathbf{q}_{L,1,h}},e_{\mathbf{q}_{R,1,h}},e_{u_{2,h}},\boldsymbol{\phi}_{L,1},\boldsymbol{\phi}_{R,1})\\
	=&\frac{1}{2}\int_{0}^{T}\frac{\partial}{\partial t}\|e_{u_{2,h}}\|^2_{L^2(\Omega)}dt+\frac{1}{\tau_{max}}\int_{0}^{T}(e_{u_{2,h}}-e_{u_{1,h}},e_{u_{2,h}})_{\Omega}dt\\
	&+\int_{0}^{T}\left (\int_0^1\cos\left (\frac{2-\alpha}{2}\pi\right )|e_{q^x_{L,1,h}}(\cdot,y)|^2_{H^{(2-\alpha)/2}(0,1)}dy\right .\\
	&\qquad\quad+\int_0^1\cos\left (\frac{2-\beta}{2}\pi\right )|e_{q^y_{L,1,h}}(x,\cdot)|^2_{H^{(2-\beta)/2}(0,1)}dx\\
	&\qquad\quad+\int_0^1\cos\left (\frac{2-\alpha}{2}\pi\right )|e_{q^x_{R,1,h}}(\cdot,y)|^2_{H^{(2-\alpha)/2}(0,1)}dy\\
	&\qquad\quad\left .+\int_0^1\cos\left (\frac{2-\beta}{2}\pi\right )|e_{q^y_{R,1,h}}(x,\cdot)|^2_{H^{(2-\beta)/2}(0,1)}dx\right )dt.
	\end{aligned}
	\end{equation*}
	Combining \eqref{equBin1eq0}, the above equalities, and the fact
	\begin{equation*}
	\int_{0}^{T}(e_{u_{1,h}}-e_{u_{2,h}},e_{u_{1,h}})_{\Omega}+(e_{u_{2,h}}-e_{u_{1,h}},e_{u_{2,h}})_{\Omega}dt=\int_{0}^{T}\|e_{u_{1,h}}-e_{u_{2,h}}\|^2_{L^2(\Omega)}dt,
	\end{equation*}
	we get \eqref{equressta}.
\end{proof}
\subsection{Error estimates}
To do the error estimations, we define the orthogonal projection operators $\mathbb{P}_1:H^1_0(\Omega)\rightarrow V_h $, $\mathbb{S}_{1}:H^{(2-\alpha)/2}(\Omega)\times H^{(2-\beta)/2}(\Omega)\rightarrow (V_h)^2$,
i.e., for all the cells $ I_{i,j} $, $i=0,1\cdots,N-1$, $j=0,1\cdots,N-1$, we have
\begin{equation*}
\begin{aligned}
&\left (\mathbb{P}_1u-u,v\right )_{I_{i,j}}=0~~~~~\forall ~v\in P^2_k(I_{i,j}),\\
&(\mathbb{S}_{1}\mathbf{u}-\mathbf{u},\mathbf{v})_{I_{i,j}}=0~~~\forall ~\mathbf{v}\in (P^2_k(I_{i,j}))^2,
\end{aligned}
\end{equation*}
where $P_k^2(I)$ is the space of $k$-th order polynomials in two variables on the 2-simplex $I$.
Similarly, we define the orthogonal projection operators $ \mathbb{P}_2:H^1_0(\Omega)\rightarrow W_h $, $ \mathbb{S}_{2}:H^{(2-\alpha)/2}(\Omega)\times H^{(2-\beta)/2}(\Omega)\rightarrow (W_h)^2$, 
i.e., for all the cells $ \bar{I}_{i,j} $, $i=0,1\cdots,N$, $j=0,1\cdots,N$, there are
\begin{equation*}
\begin{aligned}
&\left (\mathbb{P}_2u-u,v\right )_{\bar{I}_{i,j}}=0~~~~~\forall ~v\in P^2_k(\bar{I}_{i,j}),\\
&(\mathbb{S}_{2}\mathbf{u}-\mathbf{u},\mathbf{v})_{\bar{I}_{i,j}}=0~~~\forall ~\mathbf{v}\in (P^2_k(\bar{I}_{i,j}))^2.
\end{aligned}
\end{equation*}

The following inverse and trace inequalities will also be used.
\begin{lemma}[\cite{Brenner2010}]
	For any $ \mathbf{v}\in (P^2_k(I))^2 $ and $ w\in P^2_k(I) $, there exist positive constants $C$, such that
	\begin{equation*}
	\begin{aligned}
	&\|\mathbf{n}\cdot\mathbf{v}\|^2_{L^2(\partial I)}\leq Ch^{-1}\|\mathbf{v}\|^2_{L^2(I)},\\
	&\|w\|^2_{L^2(\partial I)}\leq Ch^{-1}\|w\|^2_{L^2(I)},\\
	&\|w\|_{H^l(I)}\leq Ch^{m-l}\|w\|_{H^m(I)},
	\end{aligned}
	\end{equation*}
	where $0\leq m\leq l$.
\end{lemma}
\begin{theorem}\label{errorth}
	Numerical schemes (\ref{equBin1sty}) and (\ref{equBin2sty}) satisfy the $L^2$ error estimate
	\begin{equation}\label{errorestimate}
	\begin{aligned}
	\|u-u_{1,h}\|_{L^2(
		\Omega)}+\|u-u_{2,h}\|_{L^2(\Omega)}\leq Ch^k,
	\end{aligned}
	\end{equation}
	where $u$ is the exact solution of Eq. \eqref{equation2D}, and $u_{1,h}$ and $u_{2,h}$  are the solutions of the numerical schemes (\ref{equBin1sty}) and (\ref{equBin2sty}); and $k$ is the order of polynomial.
\end{theorem}
\begin{proof}
	Denote
	\begin{equation*}
	\begin{aligned}
	&e_{u_{1}}=u(x,y,t)-u_{1,h}(x,y,t),\ e_{\mathbf{q}_{L,1}}=\mathbf{q}_{L}-\mathbf{q}_{L,1,h},\ e_{\mathbf{q}_{R,1}}=\mathbf{q}_{R}-\mathbf{q}_{R,1,h},\\
	&e_{u_{2}}=u(x,y,t)-u_{2,h}(x,y,t),\ e_{\mathbf{q}_{L,2}}=\mathbf{q}_{L}-\mathbf{q}_{L,2,h},\ e_{\mathbf{q}_{R,2}}=\mathbf{q}_{R}-\mathbf{q}_{R,2,h}.
	\end{aligned}
	\end{equation*}
	According to \eqref{equBin1sty} and \eqref{equBin2sty}, we obtain
	\begin{equation}\label{equErest1}
	\begin{aligned}
	B_1(e_{u_1},e_{u_2},e_{\mathbf{q}_{L,2}},e_{\mathbf{q}_{R,2}},v_1,\boldsymbol{\phi}_{L,2},\boldsymbol{\phi}_{R,2})=0,\\
	B_2(e_{u_2},e_{u_1},e_{\mathbf{q}_{L,1}},e_{\mathbf{q}_{R,1}},v_2,\boldsymbol{\phi}_{L,1},\boldsymbol{\phi}_{R,1})=0,
	\end{aligned}
	\end{equation}
	for all $ (v_1,\boldsymbol{\phi}^{i,j}_{L,1}, \boldsymbol{\phi}^{i,j}_{R,1})\in H^1(0,T;V^h)\times(L^2(0,T;V^h))^2\times (L^2(0,T;V^h))^2$ and
	$ (v_2,\boldsymbol{\phi}^{i,j}_{L,2},\boldsymbol{\phi}^{i,j}_{R,2})\in H^1(0,T;W^h)\times(L^2(0,T;W^h))^2\times (L^2(0,T;W^h))^2$ with ${\rm supp}( \boldsymbol{\phi}^{i,j}_{L,1})$ and ${\rm  supp}( \boldsymbol{\phi}^{i,j}_{R,1})\subset I_{i,j}$, and ${\rm supp}( \boldsymbol{\phi}^{i,j}_{L,2})$ and $ {\rm supp}( \boldsymbol{\phi}^{i,j}_{R,2})\subset \bar{I}_{i,j}$. Then we take
	\begin{equation*}
	\begin{aligned}
	&v_1=\mathbb{P}_1u-u_{1,h},\ \boldsymbol{\phi}^{i,j}_{L,1}=(\mathbb{S}_{1}\mathbf{q}_{L}-\mathbf{q}_{L,1,h})\chi_{I_{i,j}},\ \boldsymbol{\phi}^{i,j}_{R,1}=(\mathbb{S}_{1}\mathbf{q}_{R}-\mathbf{q}_{R,1,h})\chi_{I_{i,j}},\\
	&v_2=\mathbb{P}_2u-u_{2,h},\
	\boldsymbol{\phi}^{i,j}_{L,2}=(\mathbb{S}_{2}\mathbf{q}_{L}-\mathbf{q}_{L,2,h})\chi_{\bar{I}_{i,j}},\
	\boldsymbol{\phi}^{i,j}_{R,2}=(\mathbb{S}_{2}\mathbf{q}_{R}-\mathbf{q}_{R,2,h})\chi_{\bar{I}_{i,j}}
	\end{aligned}
	\end{equation*}
	in \eqref{equErest1} and denote
	\begin{equation*}
	\begin{aligned}
	&\boldsymbol{\phi}'_{L,1}=\mathbb{S}_{1}\mathbf{q}_{L}-\mathbf{q}_{L,1,h},\ \boldsymbol{\phi}'_{R,1}=\mathbb{S}_{1}\mathbf{q}_{R}-\mathbf{q}_{R,1,h},\\
	&\boldsymbol{\phi}'_{L,2}=\mathbb{S}_{2}\mathbf{q}_{L}-\mathbf{q}_{L,2,h},\
	\boldsymbol{\phi}'_{R,2}=\mathbb{S}_{2}\mathbf{q}_{R}-\mathbf{q}_{R,2,h}.
	\end{aligned}
	\end{equation*}
	By simple calculations, we obtain
	\begin{align}
	\label{equbin1veqve}
	B_1(v_1,v_2,\boldsymbol{\phi}'_{L,2},\boldsymbol{\phi}'_{R,2},v_1,\boldsymbol{\phi}_{L,2},\boldsymbol{\phi}_{R,2})&=B_1(v^e_1,v^e_2,\boldsymbol{\phi}^e_{L,2},\boldsymbol{\phi}^e_{R,2},v_1,\boldsymbol{\phi}_{L,2},\boldsymbol{\phi}_{R,2}),\\
	\label{equbin2veqve}
	B_2(v_2,v_1,\boldsymbol{\phi}'_{L,1},\boldsymbol{\phi}'_{R,1},v_2,\boldsymbol{\phi}_{L,1},\boldsymbol{\phi}_{R,1})&=B_2(v^e_2,v^e_1,\boldsymbol{\phi}^e_{L,1},\boldsymbol{\phi}^e_{R,1},v_2,\boldsymbol{\phi}_{L,1},\boldsymbol{\phi}_{R,1}),
	\end{align}
	where
	\begin{equation*}
	\begin{aligned}
	&v^e_1=\mathbb{P}_1u-u,\ \boldsymbol{\phi}^e_{L,1}=\mathbb{S}_{1}\mathbf{q}_{L}-\mathbf{q}_{L},\ \boldsymbol{\phi}^e_{R,1}=\mathbb{S}_{1}\mathbf{q}_{R}-\mathbf{q}_{R},\\
	&v^e_2=\mathbb{P}_2u-u,\
	\boldsymbol{\phi}^e_{L,2}=\mathbb{S}_{2}\mathbf{q}_{L}-\mathbf{q}_{L},\
	\boldsymbol{\phi}^e_{R,2}=\mathbb{S}_{2}\mathbf{q}_{R}-\mathbf{q}_{R}.
	\end{aligned}
	\end{equation*}
	For \eqref{equbin1veqve}, according to the proof of Theorem \ref{thml2stab}, there exists
	\begin{equation*}
	\begin{aligned}
	&B_1(v_1,v_2,\boldsymbol{\phi}'_{L,2},\boldsymbol{\phi}'_{R,2},v_1,\boldsymbol{\phi}_{L,2},\boldsymbol{\phi}_{R,2})\\
	=&\frac{1}{2}\int_{0}^{T}\frac{\partial}{\partial t}\|v_1\|^2_{L^2(\Omega)}dt+\frac{1}{\tau_{max}}\int_{0}^{T}(v_1-v_2,v_1)_{\Omega}dt\\
	&+\int_{0}^{T}\left (\int_0^1\cos\left (\frac{2-\alpha}{2}\pi\right )|\phi'^x_{L,2}(\cdot,y)|^2_{H^{(2-\alpha)/2}(0,1)}dy\right .\\
	&\qquad\quad+\int_0^1\cos\left (\frac{2-\beta}{2}\pi\right )|\phi'^y_{L,2}(x,\cdot)|^2_{H^{(2-\beta)/2}(0,1)}dx\\
	&\qquad\quad+\int_0^1\cos\left (\frac{2-\alpha}{2}\pi\right )|\phi'^x_{R,2}(\cdot,y)|^2_{H^{(2-\alpha)/2}(0,1)}dy\\
	&\qquad\quad\left .+\int_0^1\cos\left (\frac{2-\beta}{2}\pi\right )|\phi'^y_{R,2}(x,\cdot)|^2_{H^{(2-\beta)/2}(0,1)}dx\right )dt.
	\end{aligned}
	\end{equation*}
	As for the right side of \eqref{equbin1veqve}, we have
	\begin{equation*}
	\begin{aligned}
	&B_1(v^e_1,v^e_2,\boldsymbol{\phi}^e_{L,2},\boldsymbol{\phi}^e_{R,2},v_1,\boldsymbol{\phi}_{L,2},\boldsymbol{\phi}_{R,2})\\
	=&\sum_{i=0,j=0}^{i=N-1,j=N-1}\int_{0}^{T}\left(\frac{\partial v^e_1}{\partial t},v_1\right)_{I_{i,j}}+(\boldsymbol{\phi}^e_{L,2}+\boldsymbol{\phi}^e_{R,2},\nabla v_1)_{I_{i,j}}\\
	&\qquad\qquad\qquad\quad-\left(\mathbf{n}\cdot(\boldsymbol{\phi}^e_{L,2}+\boldsymbol{\phi}^e_{R,2}),v_1\right)_{\partial I_{i,j}}dt\\
	&+\sum_{i=0,j=0}^{i=N,j=N}\int_{0}^{T}\left (\mathbf{D}^{(2-\alpha)/2,(2-\beta)}_L\boldsymbol{\phi}^e_{L,2},\mathbf{D}^{(2-\alpha)/2,(2-\beta)}_R\boldsymbol{\phi}^{i,j}_{L,2}\right )_{\Omega}\\
	&\qquad\qquad\qquad+(v^e_1,\nabla\boldsymbol{\phi}^{i,j}_{L,2})_{\bar{I}_{i,j}}
	-(v^e_1,\mathbf{n}\cdot \boldsymbol{\phi}^{i,j}_{L,2})_{\partial \bar{I}_{i,j}}\\
	&\qquad\qquad\qquad+\left (\mathbf{D}^{(2-\alpha)/2,(2-\beta)}_R\boldsymbol{\phi}^e_{R,2},\mathbf{D}^{(2-\alpha)/2,(2-\beta)}_L\boldsymbol{\phi}^{i,j}_{R,2}\right )_{\Omega}\\
	&\qquad\qquad\qquad+(v^e_1,\nabla \boldsymbol{\phi}^{i,j}_{R,2})_{\bar{I}_{i,j}}-(v^e_1,\mathbf{n}\cdot \boldsymbol{\phi}^{i,j}_{R,2})_{\partial \bar{I}_{i,j}}dt\\
	&+\frac{1}{\tau_{max}}\int_{0}^{T}(v^e_1-v^e_2,v_1)_{\Omega}dt\\
	=&\uppercase\expandafter{\romannumeral1}+\uppercase\expandafter{\romannumeral2}+\uppercase\expandafter{\romannumeral3}+\uppercase\expandafter{\romannumeral4}+\uppercase\expandafter{\romannumeral5},
	\end{aligned}
	\end{equation*}
	where
	\begin{equation*}
	\begin{aligned}
	\uppercase\expandafter{\romannumeral1}=&\int_{0}^{T}\left(\frac{\partial v^e_1}{\partial t},v_1\right)_\Omega dt,\\
	\uppercase\expandafter{\romannumeral2}=&\int_{0}^{T}\left (\mathbf{D}^{(2-\alpha)/2,(2-\beta)/2}_L\boldsymbol{\phi}^e_{L,2},\mathbf{D}^{(2-\alpha)/2,(2-\beta)/2}_R\boldsymbol{\phi}_{L,2}^{i,j}\right )_\Omega\\
	&\qquad+\left (\mathbf{D}^{(2-\alpha)/2,(2-\beta)/2}_R\boldsymbol{\phi}^e_{R,2},\mathbf{D}^{(2-\alpha)/2,(2-\beta)/2}_L\boldsymbol{\phi}_{R,2}^{i,j}\right )_{\Omega}dt,\\
	\uppercase\expandafter{\romannumeral3}=&\sum_{i=0,j=0}^{i=N-1,j=N-1}\int_{0}^{T}(\boldsymbol{\phi}^e_{L,2}+\boldsymbol{\phi}^e_{R,2},\nabla v_1)_{I_{i,j}}-\left(\mathbf{n}\cdot(\boldsymbol{\phi}^e_{L,2}+\boldsymbol{\phi}^e_{R,2}),v_1\right)_{\partial I_{i,j}}dt,\\
	\uppercase\expandafter{\romannumeral4}=&\sum_{i=0,j=0}^{i=N,j=N}\int_{0}^{T}(v^e_1,\nabla\boldsymbol{\phi}^{i,j}_{L,2})_{\bar{I}_{i,j}}-(v^e_1,\mathbf{n}\cdot \boldsymbol{\phi}^{i,j}_{L,2})_{\partial \bar{I}_{i,j}}\\
	&\qquad\quad\qquad+(v^e_1,\nabla\boldsymbol{\phi}^{i,j}_{R,2})_{\bar{I}_{i,j}}-(v^e_1,\mathbf{n}\cdot \boldsymbol{\phi}^{i,j}_{R,2})_{\partial \bar{I}_{i,j}}dt,\\
	\uppercase\expandafter{\romannumeral5}=&\frac{1}{\tau_{max}}\int_{0}^{T}(v^e_1-v^e_2,v_1)_{\Omega}dt.
	\end{aligned}
	\end{equation*}
	By Cauchy-Schwarz inequality and standard approximation theory, we obtain
	\begin{equation*}
	\uppercase\expandafter{\romannumeral1}\leq Ch^{2k+2}+C\int_{0}^{T}\|v_1(s)\|^2_{L^2(\Omega)}dt.
	\end{equation*}
	Using Cauchy-Schwarz inequality and Lemma \ref{lemequnorm2D}, we have
	\begin{equation*}
	\begin{aligned}
	\uppercase\expandafter{\romannumeral2}
	\leq&\int_{0}^{T}C\|\phi^{e,x}_{L,2}\|_{H^{(2-\alpha)/2}(\Omega)}\|\phi^x_{L,2}\|_{H^{(2-\alpha)/2}(\Omega)}+C\|\phi^{e,y}_{L,2}\|_{H^{(2-\beta)/2}(0,1)}\|\phi^y_{L,2}\|_{H^{(2-\beta)/2}(\Omega)}\\
	&+C\|\phi^{e,x}_{R,2}\|_{H^{(2-\alpha)/2}(0,1)}\|\phi^x_{R,2}\|_{H^{(2-\alpha)/2}(\Omega)}+C\|\phi^{e,y}_{R,2}\|_{H^{(2-\beta)/2}(0,1)}\|\phi^y_{R,2}\|_{H^{(2-\beta)/2}(\Omega)}dt.
	\end{aligned}
	\end{equation*}
	According to the property of the  standard projection, the standard properties of interpolation spaces \cite{Adams1975}, and Young's inequality, we obtain
	\begin{equation*}
	\begin{aligned}
	\uppercase\expandafter{\romannumeral2}\leq&\frac{C}{\epsilon_1}h^{2k+\alpha}+\epsilon_1\int_{0}^{T}\|\phi^{x}_{L,2}\|^2_{H^{(2-\alpha)/2}(\Omega)}dt\\
	&+\frac{C}{\epsilon_2}h^{2k+\alpha}+\epsilon_2\int_{0}^{T}\|\phi^{x}_{R,2}\|^2_{H^{(2-\alpha)/2}(\Omega)}dt\\
	&+\frac{C}{\epsilon_3}h^{2k+\beta}+\epsilon_3\int_{0}^{T}\|\phi^{y}_{L,2}\|^2_{H^{(2-\beta)/2}(\Omega)}dt\\
	&+\frac{C}{\epsilon_4}h^{2k+\beta}+\epsilon_4\int_{0}^{T}\|\phi^{y}_{R,2}\|^2_{H^{(2-\beta)/2}(\Omega)}dt.\\
	\end{aligned}
	\end{equation*}
	By means of the standard projection property, inverse inequality, and trace inequality, there is
	\begin{equation*}
	\uppercase\expandafter{\romannumeral3}\leq Ch^{2k}+C\int_{0}^{T}\|v_1(s)\|^2_{L^2(\Omega)}dt.
	\end{equation*}
	Again,	by standard projection property, inverse and trace inequalities, we get
	\begin{equation*}
	\uppercase\expandafter{\romannumeral4}\leq \frac{C}{\epsilon_5}h^{2k}+\epsilon_5\int_0^T(\|\boldsymbol{\phi}_L\|^2_{L^2(\Omega)}+\|\boldsymbol{\phi}_R\|^2_{L^2(\Omega)})dt.
	\end{equation*}
	Similarly
	\begin{equation*}
	\uppercase\expandafter{\romannumeral5}\leq Ch^{2k+2}+\int_{0}^T \|v_1(s)\|^2_{L^2(\Omega)}ds.
	\end{equation*}
	So for sufficiently small $\epsilon_1$, $\epsilon_2$, $\epsilon_3$, $\epsilon_4$, $\epsilon_5$, we have
	\begin{equation}\label{equestof1}
	\|v_1(t)\|^2_{L^2(\Omega)}+\frac{1}{\tau_{max}}\int_{0}^{T}(v_1-v_2,v_1)_{\Omega}dt\leq Ch^{2k}+C\int_{0}^T \|v_1(s)\|^2_{L^2(\Omega)}ds,
	\end{equation}
Following the same procedure, one can get
	\begin{equation}\label{equestof2}
	\|v_2(t)\|^2_{L^2(\Omega)}+\frac{1}{\tau_{max}}\int_{0}^{T}(v_2-v_1,v_2)_{\Omega}dt\leq Ch^{2k}+C\int_{0}^T \|v_2(s)\|^2_{L^2(\Omega)}ds.
	\end{equation} Combining \eqref{equestof1}, \eqref{equestof2}, Gronwall's inequality, and the standard projection property leads to the desired bound.
\end{proof}
\begin{remark}
	The above discretization and theoretical results are for two dimensional cases, which naturally work for the one dimensional case 
	\begin{equation}\label{equation1D}
	\left\{
	\begin{aligned}
	&\frac{\partial u(x,t)}{\partial t}-d(~_{-\infty}D_{x}^{\alpha}u+~_{x}D_{\infty}^{\alpha}u)=f(x,t)~~~(x,t)\in \Omega\times[0,T],\\
	&u(x,0)=g(x)~~~~~~~~~~~~~~~x\in \Omega,\\
	&u(x,t)=0~~~~~~~~~~~~~~~~~~~x\in \mathbb{R}\backslash \Omega, ~~t\in [0,T],
	\end{aligned}
	\right.
	\end{equation}
    where $d>0$, 
    and there exists the error estimate
	\begin{equation*}
	\begin{aligned}
	\|u-u_{1,h}\|_{L^2(\Omega)}+\|u-u_{2,h}\|_{L^2(\Omega)}\leq Ch^k,
	\end{aligned}
	\end{equation*}
	where $u$ denotes the exact solution, $u_{1,h}$, $u_{2,h}$ are the solutions of the central LDG scheme for Eq. \eqref{equation1D}, and $k$ is the order of polynomial.
\end{remark}
\begin{remark}
	The error estimate of Theorem \ref{errorth} is sub-optimal, the main reasons of which include two aspects. One is that the property of the projection can't  make the terms $(\boldsymbol{\phi}^e_{L,2}+\boldsymbol{\phi}^e_{R,2},\nabla v_1)_{I_{i,j}}$ and $(v^e_1,\nabla(\boldsymbol{\phi}^{i,j}_{L,2}+\boldsymbol{\phi}^{i,j}_{R,2}))_{\bar{I}_{i,j}}$ to be zero since they involve the inner products of piecewise polynomials on two different sets of meshes, and the other is that the lower regularity of $\mathbf{q}_L$ and $\mathbf{q}_R$ in the estimation of $\uppercase\expandafter{\romannumeral2} $ makes the convergence rate to be $\mathcal{O}(h^{k+\alpha/2})$. However, numerically the optimal $(k+1)$-th order accuracies  are observed. 
	 And the same situation occurs for Eq.  \eqref{equation1D}.
\end{remark}
\section{Efficient numerical implementation}
Because of the nonlocality and singularity of the fractional operators, to effectively approximate the inner product $\left (~_0D^{(\alpha-1)/2}_xu,~_xD^{(\alpha-1)/2}_1v\right )$ becomes a key issue. We first give the following lemma.
\begin{lemma}[\cite{Podlubny1999}]\label{lemleftCR}
	If $ 0<\alpha<1 $, then
	\begin{equation}
	\frac{\partial}{\partial x}~_aD^{\alpha-1}_xu=~_aD^{\alpha-1}_x\frac{\partial}{\partial x}u+f(a)\frac{(x-a)^{-\alpha}}{\Gamma(1-\alpha)}
	\end{equation}
	and
	\begin{equation}
	-\frac{\partial}{\partial x}~_xD^{\alpha-1}_bu=-~_xD^{\alpha-1}_b\frac{\partial}{\partial x}u+f(b)\frac{(b-x)^{-\alpha}}{\Gamma(1-\alpha)}.
	\end{equation}
\end{lemma}

Let $u\in P^2_k(I_{iu,j})$ for $(x,y)\in I_{iu,j}$ and  $v\in P^2_k(I_{iv,j})$ for $(x,y)\in I_{iv,j}$, and $u=0$ for $(x,y)\in \Omega \backslash I_{iu,j}$  and $v=0$ for $(x,y)\in \Omega \backslash I_{iv,j}$, where $0\leq iu\leq iv \leq N-1$. Then, there is
\begin{equation*}
\begin{aligned}
&\left(~_0D^{(\alpha-1)/2}_xu,~_xD^{(\alpha-1)/2}_1v\right)_{\Omega}\\
=&\left(\frac{1}{\Gamma(1-(\alpha-1)/2)}\right)^2\int_{0}^{1}\int_0^1\frac{\partial}{\partial x}\left(\int_0^x(x-\xi)^{(\alpha-1)/2-1}u(\xi,y)d\xi \right)\\
&\qquad\qquad\quad\qquad\quad\qquad\quad\times\frac{\partial}{\partial x}\left(-\int_x^1(\xi-x)^{(\alpha-1)/2-1}v(\xi,y)d\xi\right) dxdy\\
=&\left(\frac{1}{\Gamma(1-(\alpha-1)/2)}\right)^2\int_{y_j}^{y_{j+1}}\int_{x_{iu}}^{x_{iv+1}}\frac{\partial}{\partial x}\left(\int_0^x(x-\xi)^{(\alpha-1)/2-1}u(\xi,y)d\xi \right)\\
&\qquad\qquad\quad\qquad\quad\qquad\quad\times\frac{\partial}{\partial x}\left(-\int_x^1(\xi-x)^{(\alpha-1)/2-1}v(\xi,y)d\xi\right) dxdy.
\end{aligned}
\end{equation*}
According to Lemma \ref{lemleftCR}, for some fixed integral point $\bar{x}>x_{iu+1}$, we have
\begin{equation}
\begin{aligned}
&\left .\frac{\partial}{\partial x}\left(\int_0^x(x-\xi)^{(\alpha-1)/2-1}u(\xi,y)d\xi \right)\right |_{x=\bar{x}}\\
=&\frac{\partial}{\partial x}\left. \left(\int_{x_{iu}}^{x}(x-\xi)^{(\alpha-1)/2-1}\tilde{u}(\xi,y)d\xi-\int_{x_{iu+1}}^{x}(x-\xi)^{(\alpha-1)/2-1}\tilde{u}(\xi,y)d\xi \right)\right |_{x=\bar{x}}\\
=&\int_{x_{iu}}^{\bar{x}}(x-\xi)^{(\alpha-1)/2-1}\frac{\partial}{\partial \xi}\tilde{u}(\xi,y)d\xi-\tilde{u}(x_{iu},y)\frac{(\bar{x}-x_{iu})}{\Gamma((\alpha-1)/2)}\\
&-\left(\int_{x_{iu+1}}^{x}(x-\xi)^{(\alpha-1)/2-1}\frac{\partial}{\partial \xi}\tilde{u}(\xi,y)d\xi-\tilde{u}(x_{iu+1},y)\frac{(\bar{x}-x_{iu+1})}{\Gamma((\alpha-1)/2)}\right),
\end{aligned}
\end{equation}
where $\tilde{u}$ is the smooth extension of $u$. As for $\bar{x}\in(x_{iu},x_{iu+1})$, it can be directly calculated.

The term $ \frac{\partial}{\partial x}\left(-\int_x^1(\xi-x)^{(\alpha-1)/2-1}v(\xi,y)d\xi\right) $ can be similarly calculated. So, by the above techniques, the inner product  $\left (~_0D^{(\alpha-1)/2}_xu,~_xD^{(\alpha-1)/2}_1v\right )$ can be efficiently calculated using numerical integration.

\section{Numerical experiments}
\label{}
To demonstrate the performance of our proposed schemes, we provide one-  and two-dimensional examples to verify the accuracy. Here, we define the errors as
\begin{equation*}
E_{i,h}=u(t_n)-u^n_{i,h}, ~~i=1,2,
\end{equation*}
where $u^n_{1,h}$ and $u^n_{2,h}$ mean the numerical solutions of $u$ at time $t_n$ on two sets of meshes $\{I_{i,j}\}$ and $\{\bar{I}_{i,j}\}$ with mesh size $h$, respectively.

\begin{example}\label{example1}
	We consider the one-dimensional problem  \eqref{equation1D} with the exact solution
	\begin{equation*}
	u(x,t)=e^{2t}x^3(1-x)^3,
	\end{equation*}
	so the initial data is
	\begin{equation*}
	g(x)=x^3(1-x)^3
	\end{equation*}
	and the source term is
	\begin{equation*}
	f(x,t)=2e^{2t}x^3(1-x)^3-e^{2t}\left (~_0D^{\alpha}_xx^3(1-x)^3+~_xD^{\alpha}_1x^3(1-x)^3\right ),
	\end{equation*}
	which can be calculated by numerical integration.
	Here we set $T=0.1$. To investigate the convergence in space and reduce the influence from temporal discretization, we take $\tau_{max}=0.1h^{\alpha}$ and $\tau=0.1\tau_{max}$ when the order of approximation polynomial $k=1$, and $\tau_{max}=0.005h^{\alpha}$ and $\tau=0.01\tau_{max}$ when $k=2$ respectively, and Tables \ref{tab:1dN1} and  \ref{tab:1dN2} provide the numerical errors and convergence rates when $\alpha=1.1, 1.5, 1.9$, which show that our method can achieve the optimal $(k+1)$-th convergence order.
	\begin{table}[htpb]
		\caption{$L^2$ errors and convergence rates at $T=0.1$ with the order of polynomial $k=1$}
		\label{tab:1dN1}
		\begin{tabular}{c|c|cccccc}
			\hline
			$\alpha$&$1/h$&          8 &         16 &         32 &         64 &        128 &        256 \\
			\hline
			&$E_{1,h}$ &  3.860E-04 &  7.364E-05 &  1.366E-05 &  2.880E-06 &  6.793E-07 &  1.706E-07 \\
			
			1.1 & &  rate          &    2.3901  &    2.4308  &    2.2457  &    2.0837  &    1.9938  \\
			
			&$E_{2,h}$ &  3.852E-04 &  7.371E-05 &  1.367E-05 &  2.880E-06 &  6.793E-07 &  1.706E-07 \\
			
			& &rate            &    2.3856  &    2.4311  &    2.2464  &    2.0840  &    1.9939  \\
			\hline
			&$E_{1,h}$ &  2.259E-04 &  5.352E-05 &  1.223E-05 &  2.852E-06 &  6.834E-07 &  1.668E-07 \\
			
			1.5 & & rate           &    2.0774  &    2.1301  &    2.0999  &    2.0614  &    2.0347  \\
			
			&$E_{2,h}$ &  2.276E-04 &  5.418E-05 &  1.228E-05 &  2.855E-06 &  6.835E-07 &  1.668E-07 \\
			
			& & rate           &    2.0705  &    2.1416  &    2.1044  &    2.0625  &    2.0349  \\
			\hline
			&$E_{1,h}$ &  2.483E-04 &  5.896E-05 &  1.442E-05 &  3.526E-06 &  8.603E-07 &  2.098E-07 \\
			
			1.9 & &  rate          &    2.0744  &    2.0320  &    2.0319  &    2.0349  &    2.0360  \\
			
			&$E_{2,h}$ &  2.511E-04 &  6.038E-05 &  1.452E-05 &  3.530E-06 &  8.605E-07 &  2.098E-07 \\
			
			& &  rate          &    2.0561  &    2.0565  &    2.0399  &    2.0365  &    2.0363  \\
			\hline
		\end{tabular}
		
	\end{table}
	
	\begin{table}[htpb]
		\caption{$L^2$ errors and convergence rates at $T=0.1$ with the order of polynomial $k=2$}
		\label{tab:1dN2}
		\begin{tabular}{c|c|cccccc}
			\hline
			$\alpha$& $1/h$           &          4 &          8 &         16 &         32 &         64 &        128 \\
			\hline
			&    $E_{1,h}$        &  8.821E-04 &  5.884E-05 &  1.240E-05 &  1.866E-06 &  2.169E-07 &  2.215E-08 \\
			
			1.1 &            &   rate         &    3.9062  &    2.2468  &    2.7319  &    3.1051  &    3.2916  \\
			
			&   $E_{2,h}$         &  8.832E-04 &  5.874E-05 &  1.239E-05 &  1.865E-06 &  2.168E-07 &  2.215E-08 \\
			
			&            &      rate      &    3.9102  &    2.2454  &    2.7314  &    3.1049  &    3.2915  \\
			\hline
			&   $E_{1,h}$         &  4.765E-04 &  3.675E-05 &  5.296E-06 &  7.023E-07 &  8.393E-08 &  9.691E-09 \\
			
			1.5 &            &  rate          &    3.6969  &    2.7947  &    2.9148  &    3.0648  &    3.1144  \\
			
			&  $E_{2,h}$          &  4.785E-04 &  3.651E-05 &  5.279E-06 &  7.014E-07 &  8.388E-08 &  9.688E-09 \\
			
			&            &      rate      &    3.7122  &    2.7899  &    2.9120  &    3.0639  &    3.1140  \\
			\hline
			&  $E_{1,h}$          &  2.729E-04 &  2.831E-05 &  3.343E-06 &  3.844E-07 &  4.377E-08 &  5.109E-09 \\
			
			1.9 &            &  rate          &    3.2693  &    3.0818  &    3.1205  &    3.1347  &    3.0989  \\
			
			&   $E_{2,h}$         &  2.766E-04 &  2.786E-05 &  3.301E-06 &  3.816E-07 &  4.359E-08 &  5.098E-09 \\
			
			&            &      rate      &    3.3114  &    3.0773  &    3.1127  &    3.1300  &    3.0962  \\
			\hline
		\end{tabular}
	\end{table}
\end{example}

\begin{example}
	Consider the two-dimensional problem \eqref{equation2D} with the exact solution
	\begin{equation*}
	u(x,y,t)=1000e^tx^3(1-x)^3y^3(1-y)^3,
	\end{equation*}
	so the initial data is
	\begin{equation*}
	g(x,y)=1000x^3(1-x)^3y^3(1-y)^3,
	\end{equation*}
	and taking $d_1=-1/(2\cos(\alpha\pi/2))$ and $d_2=-1/(2\cos(\beta\pi/2))$ leads to the source term
	\begin{equation*}
	\begin{aligned}
	f(x,y,t)=&1000e^tx^3(1-x)^3y^3(1-y)^3\\
	&-1000d_1e^t\left (~_0D^{\alpha}_xx^3(1-x)^3y^3(1-y)^3+~_xD^{\alpha}_1x^3(1-x)^3y^3(1-y)^3\right )\\
	&-1000d_2e^t\left (~_0D^{\beta}_yx^3(1-x)^3y^3(1-y)^3+~_yD^{\beta}_1x^3(1-x)^3y^3(1-y)^3\right ),
	\end{aligned}
	\end{equation*}
	which can be calculated by numerical integration.
	Here we set $T=0.1$ and the order of polynomial $k=1$.  We take $\tau_{max}=0.02h^{\alpha}$ and $\tau=0.1\tau_{max}$ to investigate the convergence in space and reduce the influence from temporal discretization. Table \ref{tab:2dN1} provides the convergence orders when $\alpha$ and $\beta$ equal to 1.1, 1.5, 1.9, respectively,  which shows our method can achieve the optimal $(k+1)$-th convergence order and validates the effectiveness of our algorithm.
	
	\begin{table}[htpb]
		\caption{$L^2$ errors and convergence rates at $T=0.1$ with the order of polynomial $k=1$}
		\label{tab:2dN1}
		\begin{tabular}{c|c|ccccc}
			\hline
			$(\alpha,\beta)$&$1/h$            &          4 &          8 &         12 &         16 &         20 \\
			\hline
			&$E_{1,h}$            &  2.190E-02 &  7.518E-03 &  2.751E-03 &  1.417E-03 &  7.937E-04 \\
			
			(1.1,1.1) &            &          rate &    1.5423  &    2.4794  &    2.3064  &    2.5973  \\
			
			&$E_{2,h}$            &  2.162E-02 &  7.495E-03 &  2.743E-03 &  1.416E-03 &  7.943E-04 \\
			
			&            & rate           &    1.5282  &    2.4794  &    2.2971  &    2.5921  \\
			\hline
			&$E_{1,h}$            &  1.675E-02 &  5.516E-03 &  2.377E-03 &  1.249E-03 &  7.468E-04 \\
			
			(1.5,1.5) &            &         rate &    1.6026  &    2.0759  &    2.2378  &    2.3042  \\
			
			&$E_{2,h}$            &  1.610E-02 &  5.490E-03 &  2.373E-03 &  1.248E-03 &  7.464E-04 \\
			
			&            & rate           &    1.5522  &    2.0683  &    2.2348  &    2.3025  \\
			\hline
			&$E_{1,h}$            &  1.483E-02 &  5.475E-03 &  2.692E-03 &  1.560E-03 &  1.005E-03 \\
			
			(1.9,1.9) &            &          rate &    1.4379  &    1.7506  &    1.8973  &    1.9694  \\
			
			&$E_{2,h}$            &  1.391E-02 &  5.453E-03 &  2.690E-03 &  1.559E-03 &  1.005E-03 \\
			
			&            &  rate          &    1.3512  &    1.7428  &    1.8953  &    1.9687  \\
			\hline
		\end{tabular}
	\end{table}
	
\end{example}

\section{Conclusions}
\label{}

This paper provides the weak formulation for numerically solving the one- and two-dimensional space fractional diffusion equations by the central LDG method, which uses the informations of the solution on the overlapping cells and needn't the numerical flux any more. The numerical stability is proved and error estimates are presented. Finally, the numerical experiments are performed to verify the theoretical results and effectiveness of the schemes.


\section{Acknowledgements}
\label{}
This work was supported by the National Natural Science Foundation of China under grant no. 11671182, and the Fundamental Research Funds for the Central Universities under grants no. lzujbky-2018-ot03 and no. lzujbky-2017-ot10.

\begin{thebibliography}{1}
\expandafter\ifx\csname url\endcsname\relax
  \def\url#1{\texttt{#1}}\fi
\expandafter\ifx\csname urlprefix\endcsname\relax\def\urlprefix{URL }\fi
\expandafter\ifx\csname href\endcsname\relax
  \def\href#1#2{#2} \def\path#1{#1}\fi

\bibitem{Feynman1963118}
R.~Feynman, F.~{Vernon Jr.}, The theory of a general quantum system interacting
  with a linear dissipative system, Annals of Physics 24 (1963) 118--173.
\newblock \href {http://dx.doi.org/10.1016/0003-4916(63)90068-X}
  {\path{doi:10.1016/0003-4916(63)90068-X}}.

\bibitem{Dirac1953888}
P.~Dirac, The lorentz transformation and absolute time, Physica 19~(1-–12)
  (1953) 888--896.
\newblock \href {http://dx.doi.org/10.1016/S0031-8914(53)80099-6}
  {\path{doi:10.1016/S0031-8914(53)80099-6}}.

\end{thebibliography}


\begin{thebibliography}{00}
	
	
	
	\bibitem{Adams1975}R.A. Adams, Sobolev Spaces, Academic Press, 1975.
	
	\bibitem{Basu2012}T.S. Basu, H. Wang, A fast second-order finite difference method for space-fractional diffusion equations, Int. J. Numer. Anal. Model. 9 (2012) 658-666.
	
	
	\bibitem{Brenner2010}S.C. Brenner, L.R. Scott, The Mathematical Theory of Finite Element Methods, Springer-Verlag, 1994.
	
	\bibitem{Bu2014}W.P. Bu, Y.F. Tang, J.Y. Yang, Galerkin finite element method for two-dimensional Riesz space fractional diffusion equations, J. Comput. Phys. 276 (2014) 26-38.
	
	
	
	\bibitem{Chen2014}M.H. Chen, Y.T. Wang, X. Cheng, W.H. Deng, Second-order LOD multigrid method for multidimensional Riesz fractional diffusion equation, BIT 54 (2014) 623-647.
	
	\bibitem{Cockburn1998}B. Cockburn, C.W. Shu, The Runge-Kutta discontinuous Galerkin method for conservation laws V, J. Comput. Phys. 141 (1998) 199-224.
	
	\bibitem{Cockburn20082}B. Cockburn, C.W. Shu, The local discontinuous Galerkin method for time-dependent convection-diffusion systems, SIAM J. Numer. Anal. 35 (1998) 2440-2463.
	
	
	\bibitem{Cui2009}M.R. Cui, Compact finite difference method for the fractional diffusion equation, J. Comput. Phys. 228 (2009) 7792-7804.
	

\bibitem{Deng2018} W.H. Deng, Z.J. Zhang, High Accuracy Algorithms for the Integral-Differential Equations Governing Anomalous Diffusion, World Scientific, Singapore, 2018.

	\bibitem{Deng2013}W.H. Deng, J.S. Hesthaven, Local discontinuous Galerkin methods for fractional diffusion equations, ESAIM Math. Model. Numer. Anal. 47 (2013) 1845-1864.
	
	
	\bibitem{Ervin2006}V.J. Ervin, J.P. Roop, Variational formulation for the stationary fractional advection dispersion equation, Numer. Methods Partial Differential Equations. 22 (2006) 558-576.
	
	\bibitem{Ervin2007}V.J. Ervin, J.P. Roop, Variational solution of fractional advection dispersion equations on bounded domains in $\mathbb{R}^d$, Numer. Methods Partial Differential Equations. 23 (2007) 256-281.
	
	\bibitem{Kilbas2006}A.A. Kilbas, H.M. Srivastava, J.J. Trujillo, Theory and Applications of Fractional Differential Equations, Elsevier, 2006.
	
	\bibitem{Li2009}X.J. Li, C.J. Xu, A space-time spectral method for the time fractional diffusion equation,  SIAM J. Numer. Anal. 47 (2009) 2108-2131.
	
	\bibitem{Lions1972}J.L. Lions, E. Magenes, Non-homogeneous Boundary Value Problems and Applications,
	Vol. 1, Springer-Verlag,  1972.
	
	\bibitem{Liu2004}F. Liu, V. Anh, I. Turner, Numerical solution of the space fractional Fokker-Planck equation, J. Comput. Appl. Math. 166 (2004) 209-219.
	
	
	\bibitem{Liu2007}Y.J. Liu, C.W. Shu, E. Tadmor, M.P. Zhang, Central discontinuous Galerkin methods on overlapping cells with a nonoscillatory hierarchical reconstruction, SIAM J. Numer. Anal.  45 (2007) 2442-2467.
	
	\bibitem{Liu2008}Y.J. Liu, C.W. Shu, E. Tadmor, M.P. Zhang, $L^2$ stability analysis of the central discontinuous Galerkin method and a comparison between the central and regular discontinuous Galerkin methods, ESAIM Math. Model. Numer. Anal. 42 (2008) 593-607.
	
	
	\bibitem{Liu2011}Y.J. Liu, C.W. Shu, E. Tadmor, M.P. Zhang, Central local discontinuous Galerkin methods on overlapping cells for diffusion equations, ESAIM Math. Model. Numer. Anal. 45 (2011) 1009-1032.
	
	\bibitem{Mao2016}Z.P. Mao, J. Shen, Efficient spectral-Galerkin methods for fractional partial differential equations with variable coefficients, J. Comput. Phys. 307 (2016) 243-261.
	
	\bibitem{Metzler2000}R. Metzler, J. Klafter, Boundary value problems for fractional diffusion equations, Phys. A 278(2000) 107-125.
	
	
	\bibitem{Metzler20002}R. Metzler, J. Klafter, The random walk's guide to anomalous diffusion: a fractional dynamics approach, Phys. Rep. 339 (2000) 1-77.
	
	\bibitem{Nessyahu1990}H. Nessyahu, E. Tadmor, Non-oscillatory central differencing for hyperbolic conservation laws, J. Comput. Phys. 87 (1990) 408-463.
	
	
	\bibitem{Podlubny1999}I. Podlubny, Fractional Differential Equations, Academic Press, 1999.
	
	\bibitem{Qiu2015}L.L. Qiu, W.H. Deng, J.S. Hesthaven, Nodal discontinuous Galerkin methods for fractional diffusion equations on 2D domain with triangular meshes, J. Comput. Phys. 298 (2015) 678-694.
	
	\bibitem{Xu2014}Q.W. Xu, J.S. Hesthaven, Discontinuous Galerkin method for fractional convection-diffusion equations,  SIAM J. Numer. Anal. 52 (2014) 405-423.
	
	\bibitem{Yuste2004}S.B. Yuste, L. Acedo, K. Lindenberg, Reaction front in an A + B $\rightarrow$ C reaction-subdiffusion process, Phys. Rev. E 69 (2004).
	
	
	\bibitem{Zhang2010}H. Zhang, F. Liu, V. Anh, Galerkin finite element approximation of symmetric space-fractional partial differential equations, Appl. Math. Comput. 217 (2010) 2534-2545.
	
\end{thebibliography}


\end{document}